\documentclass[12pt,reqno]{amsart}
\usepackage{fullpage}
\usepackage{amsfonts}
\usepackage{amssymb}
\usepackage{times}
\usepackage{graphicx}
\newtheorem{thm}{Theorem}[section]

\newtheorem{prop}[thm]{Proposition}

{ \theoremstyle{remark}\newtheorem{remark}{Remark} }

\usepackage[linkcolor=blue,colorlinks=true]{hyperref}
\usepackage{amsfonts,rawfonts}
\usepackage{todonotes}
\usepackage{bbold}
\usepackage{thmtools}
\usepackage{systeme}
\usepackage{mathtools}
\usepackage{xcolor}
\usepackage{xfrac}
\usepackage{physics}
\usepackage{enumitem}
\usepackage{mathtools} 
\usepackage{tikz-cd}
\usetikzlibrary{arrows,chains,matrix,positioning,scopes}
\usepackage{bbm}
\usepackage{xr}

\usepackage{comment}
\usepackage{enumitem}
\newtheorem{theorem}{Theorem}
\newtheorem{definition}{Definition}
 
\newtheorem{lemma}{Lemma}

\newcommand{\hZ}{\hat{\mathbb{Z}} }

\def \d{\mbox{\(\,\mathrm{d}\)}}
\usepackage{mathtools}
\usepackage{tikz}
\usetikzlibrary{decorations.pathreplacing,calligraphy,backgrounds}

\begin{document}

\title[Short version of title]{Transition of type in the von Neumann algebras associated to the Connes-Marcolli $GSp_4$-system }
\author{Ismail Abouamal}
\email{abouamal@caltech.edu}
\maketitle

\begin{abstract}
We study different types of von Neumann algebras arising from the Connes-Marcolli $GSp_4$-system and show that a phase transition occurs at the level of these algebras. More precisely, we show that the type of these algebras transitions from type $\textmd{I}_\infty$ to type $\textmd{III}_1$, with this transition occurring precisely at the inverse temperature $\beta = 4$.

\end{abstract}

\section{Introduction}

\par In our previous work \cite{abouamal2022bost}, we studied the structure of all extremal $\textmd{KMS}_\beta$ states on the Connes-Marcolli $GSp_4$-system and established that a phase transition occurs at the critical inverse temperatures $\beta_{c_1} = 3$ and $\beta_{c_2} = 4$. More specifically, we showed that for $\beta >4$, every extremal $\textmd{KMS}_\beta$ state is a Gibbs state and the partition function can be expressed as the ratio of shifted Riemann zeta functions. In the range $3 < \beta \leq 4$, we proved that there exists a unique $\textmd{KMS}_\beta$ state and explicitly constructed its corresponding $\mu_\beta$-measure on the space $PGSp_{4}^+(\mathbb R) \times MSp_{4}(\mathbb A_{\mathbb Q, f})$.

\par In this paper, our focus shifts to investigating the structure of all von Neumann algebras generated by the extremal $\textmd{KMS}_\beta$ states for a given inverse temperature $\beta > 3$. In section \ref{low_temperature}, we show that the equilibrium states generate a type $\textmd{I}_{\infty}$ factor when $\beta > 4$. In section \ref{high_temperature}, we present a proof of our main result (Theorem \ref{main_theorem}) which establishes that the unique $\textmd{KMS}_\beta$ for $3 <\beta \leq 4$ is of type $\text{III}_{1}$. This amounts to proving that the action of $GSp_{4}^+(\mathbb Q)$ on $PGSp_{4}^+(\mathbb R)\times MSp_{4}(\mathbb A_{\mathbb Q,f})$ 
is of type $\textmd{III}_1$ (c.f. Definition \ref{action_typeIII}) with respect to the product measure corresponding to the unique $\textmd{KMS}_\beta$ state (c.f. \cite[Proposition 3.10.]{abouamal2022bost} for the explicit description of the product measure). The proof relies on two preliminary results. The first is the ergodicity of the action of $GSp_{4}^+(\mathbb Q)$ on $MSp_{4}(\mathbb A_{\mathbb Q,f} )$, which was established in \cite[Theorem 3.13]{abouamal2022bost}. The second component of the proof consists of showing that the action of $GSp_{4}^+(\mathbb{Q})$ on the space 
$PGSp_{4}^+(\mathbb R) \times MSp_4(\mathbb{A}_f) / GSp_{4}(\hZ)$ is of type $\textmd{III}_1$, which we prove by explicitly computing the ratio set.

We first recall some notations from \cite{abouamal2022bost}. The set of prime numbers is denoted by $\mathcal{P}$. For a given nonempty finite set of prime numbers $F \subset \mathcal{P}$, we denote by $\mathbb N (F)$ the unital multiplicative subsemigroup of $\mathbb N$ generated by $p\in F$. We denote by $A_{\mathbb Q,f}$ the ring
of finite ad\`eles of $\mathbb Q$ and set

\begin{align*}
    G &= GSp_4^+(\mathbb Q), & X &= PGSp_{4}^+ (\mathbb R) \times MSp_4(\mathbb{A}_{\mathbb {Q},f}),\\
    \Gamma &= Sp_4(\mathbb Z), & Y &= PGSp_{4}^+ (\mathbb R) \times MSp_{4}(\hZ)\subset X.
\end{align*}

The $C^*$-dynamical system we aim to study is denoted by $(\mathcal{A}, \sigma_t)$ where $\mathcal{A}$ is the completion of the algebra $C_c (\Gamma_2 \backslash G \boxtimes_{\Gamma_{2}} Y)$ in the reduced norm and the time evolution is given by

$$ \sigma_t(f)(g,y) = \lambda(g)^{it} f(g,y), \quad f\in C_c (\Gamma_2 \backslash G \boxtimes_{\Gamma_{2}} Y).$$

For a finite set of primes $F\subset \mathcal{P}$, we put 

$$ \mathbb Q_F = \prod_{p\in F} \mathbb{Q}_p, \quad 
    \mathbb Z_F = \prod_{p\in F} \mathbb Z_p,$$ and 

\begin{align*}
    X_F &= PGSp_{4}^+(\mathbb R) \times MSp_4 (\mathbb{Q}_F), \quad 
    Y_F = PGSp_{4}^+(\mathbb R) \times MSp_4 (\mathbb{Z}_F).
\end{align*}

Given a prime $p \in \mathcal{P}$ we have that

\begin{equation*}
    \{g\in MSp_{4}(\mathbb Z): \abs{\lambda(g)}=p\}=\Gamma_2 g_{1,p} \Gamma_2,
\end{equation*}
and $$\deg_{\Gamma_2}(g_{1,p})=(1+p)(1+p^2),$$

where $g_{1,p} = \textmd{diag}(1,1,p,p)$.
 We set
\begin{align}
    A_p & :=\{(\tau,x)\in PGSp_4^+(\mathbb R) \times MSp_{4}(\hZ)\, | \, x_p \in GSp_{4}(\mathbb Z_p)\},\label{label 108}\\
    B_p & :=\{(\tau,x) \in PGSp^+_{4}(\mathbb R) \times MSp_4(\hZ) \, | \, \abs{\lambda(x)}_p=p^{-1}\}. \label{label 109}
\end{align}

Denote by $\pi_{F}$ the factor map $X \rightarrow X_F$ and let $f$ be a function on $X_F$. We then define the function $f_F$ on $X$ by

\begin{equation*}
    f_F(x)= \begin{cases}
    f(\pi_{F}(x)) \quad &\text{if}\,\, x_p \in MSp_{4}(\mathbb Z_p) \,\, \text{for all}\,\, p \in F^{c}, \\
    0 \quad \,\, &\text{otherwise.}
    \end{cases}
\end{equation*}

\subsection*{Acknowledgment}
The author would like to thank his advisor Matilde Marcolli for her guidance throughout this project.

\section{Low temperature region: Type $\textmd{I}_{\infty}$ factors and Gibbs states} \label{low_temperature}

In the low temperature regime, the set of $\textmd{KMS}_\beta$ states on $(\mathcal{A}, \sigma_t)$ is parametrized by points on the space $PGSp_{4}^+(\mathbb R) \times GSp_{4}(\hZ)$. Recall from \cite{abouamal2022bost} that if $\beta > 4$, then every extremal $\textmd{KMS}_\beta$ state $\phi_\beta$ is a Gibbs state. We will now show that these states generate a family of type $I_\infty$ factors. Recall that for any $\textmd{KMS}_\beta$ state $\phi$ on the system $(\mathcal{A}, \sigma_t)$, its type corresponds to the type of the von Neumann algebra $\pi_{\phi}(\mathcal{A})''$ generated in the \textmd{GNS} representation.\\

\begin{theorem}
Let $y\in PGSp_{4}^+(\mathbb R) \times GSp_{4}(\hZ) $ and $\beta>4$. Then the $\textmd{KMS}_\beta$ state given by

\begin{equation*}
\phi_{\beta,y}(f)= \frac{\zeta(2\beta-2)\textmd{Tr} (\pi_{y}(f)e^{-\beta H_y})}{\zeta(\beta) \zeta(\beta-1)\zeta(\beta-2) \zeta(\beta-3)},\quad \forall f\in \mathcal{A}
\end{equation*}
is extremal of type $\textmd{I}_{\infty}$.
\end{theorem}

\begin{proof}

It is enough to show that the the algebra $\mathcal{A}$ associated to the Connes-Marcolli \textmd{GSp$_4$}-system generates a factor in the \textmd{GNS} representation of the state $\phi_{\beta,y}$. Consider the following representation of $\mathcal{A}$:

\begin{align*}
    \tilde{\pi}_y: \mathcal{A} &\longrightarrow \mathcal{B}(\mathcal{H}_y \otimes \mathcal{H}_y)\\
    a &\mapsto \pi_y(a) \otimes id_{\mathcal{H}_y}
\end{align*}

and denote by $ \Omega_{\beta,y}$ the unitary vector given by

\begin{equation*}
    \Omega_{\beta,y}= \zeta_{MSp_{4}(\mathbb Z),\Gamma_{_2}}(\beta)^{-1/2} \sum_{h\in \Gamma_{_2} \backslash G_y} \lambda (h)^{-\beta /2} \delta_{\Gamma_{_2} h} \otimes \delta_{\Gamma_{_2} h},
\end{equation*}

A direct computation shows that

\begin{equation*}
    \phi_{\beta,y}= \langle \tilde{\pi}_y(f) \Omega_{\beta,y},\Omega_{\beta,y} \rangle,\quad \forall f\in \mathcal{A}.
\end{equation*}
and 

\begin{equation*}
\tilde{\pi}_y(f) \Omega_{\beta,y} = \zeta_{MSp_{4}(\mathbb Z),\Gamma_2}(\beta)^{-1/2} \sum_{g,h\in \Gamma_2 \backslash G_y} \lambda (h)^{-\beta /2} f(gh^{-1},hy)\delta_{\Gamma_{_2} g} \otimes \delta_{\Gamma_{_2} h}.
\end{equation*}
By choosing $f$ with a sufficiently small support, we see that the orbit $\tilde{\pi}_y(\mathcal{A})\Omega_{\beta,y}$ is dense in $\mathcal{H}_y \otimes \mathcal{H}_y$. This shows that the $\textmd{GNS}$ representation is equivalent to the triple $(\mathcal{H}_y \otimes \mathcal{H}_y, \tilde{\pi}_y, \Omega_{\beta,y})$.

By \cite[Proposition VII.5 b)]{connes1979theorie} the commutant of $\pi_{y}(\mathcal{A})$ is generated by the right regular representation of the isotropy group $\mathcal{G}_{y}^{y}$ of the groupoid $\mathcal{G}= \Gamma_2 \backslash (G \boxtimes Y$). Since $y\in PGSp_{4}^+(\mathbb R) \times GSp_{4}(\hZ)$, the isotropy group $\mathcal{G}_{y}^{y}$ is trivial which implies that $ \pi_{y}(\mathcal{A})'=\mathbb C$. Hence

\begin{align*}
    \tilde{\pi}_y(\mathcal{A})''
    &=(\pi_{y}(\mathcal{A})' \otimes  B(\mathcal H_y))'\\
    &= B(\mathcal{H}_y) \otimes \mathbb C\\
    & \simeq B(\mathcal{H}_y)
\end{align*}

This shows that $\phi_{\beta,y}$ is an extremal state of type $I_{\infty}$.
\end{proof}

\section{Type $\textmd{III}_1$ factor state: the critical region $3<\beta \leq 4.$} \label{high_temperature}

\par Our next goal is to study the factor generated by the unique $\textmd{KMS}_\beta$ state on the $GSp_4$-system in the critical region $3<\beta \leq 4$. For $\beta > 4$, it was possible to compute the type of any Gibbs state by exhibiting an explicit formula for the $\text{GNS}$ representation (which is unique up to unitary equivalence). The approach in the critical region is less explicit. In fact, we will use a different strategy by extending the approach in \cite{bost1995hecke} and \cite{neshveyev2011neumann}. 

\par Consider now the unique \textmd{KMS} state $\phi_\beta$  on the $GSp_{4}$-system and denote by $\mu_\beta$ the corresponding $\Gamma_2$-invariant measure on $X$. We choose a $\mu_\beta$-measurable fundamental domain $F$ for the action of $\Gamma_2$ on $Y$. Then (See \cite{feldman1977ergodic} and \cite[Remark 2.3]{laca2007phase} ) the algebra $\pi_{\phi_\beta}(\mathcal{A})''$ induced by the state $\phi_\beta$ is isomorphic to the reduction of the von Neumann algebra of the $G$-orbit equivalence relation on $(X, \mu_\beta)$ by the projection $\mathbb 1_F$, that is

\begin{equation}
    \pi_{\phi_\beta}(\mathcal{A})'' \simeq \mathbb 1 _{F} (L^{\infty} (X, \mu_\beta) \rtimes G )\mathbb 1 _{F}.  \label{label 110}
\end{equation}



Consider the action of the countable group $G$ on the measure space $(X,\mathcal{F},\mu)$. We recall the following definition from \cite{krieger2006araki}.

\begin{definition} \label{action_typeIII}
The ratio set $r(G)$ of the action of $G$ on $(X,\mathcal{F},\mu)$ consists of all real numbers $\lambda \geq 0$ such that for every $\epsilon >0$ and any $A\in \mathcal{F}$ of positive measure, there exists $g\in G$ such that

$$ \mu\Big(\Big\{x\in gA\cap A: \abs{\frac{dg_{*}\mu}{\d\mu} (x) -\lambda} < \epsilon \Big \}\Big) > 0,$$

where the measure $g_*\mu$ is defined by $g_*\mu(B) = \mu(g^{-1}(B))$.
\end{definition}

The ratio set depends only on the equivalence relation $\mathcal{R} = \{(x,gx) \mid x\in X, g\in G\} \subset X \times X$ and the measure class of $\mu$ (hence we will denote the ratio by $r(\mathcal{R},\mu)$). Moreover one can show that the set  $r(\mathcal{R},\mu) \cap (0,\infty)$ is a closed subgroup of $\mathbb R^*_+$. We then have the following result (cf. \cite[Proposition 4.3.18]{sunder2012invitation}).

\begin{theorem} \label{label 111}
    Let $G$ be a countable group $G$ acting by automorphisms on a measure space $(X,\mathcal{F},\mu)$. Assume that the action of $G$  on $(X,\mathcal{F},\mu)$ is free and ergodic. Then $L^\infty(X,\mathcal{F},\mu)$ is factor of type $\textmd{III}_1$ if and only if $r(\mathcal{R},\mu) \cap (0,\infty) = \mathbb R^*_+$.
\end{theorem}
This result motivates the following definition.
\begin{definition}
    The action of $G$ on the measure space $(X,\mathcal{F},\mu)$ is said to be of type $\textmd{III}_1$ if $$r(\mathcal{R},\mu) \cap (0,\infty) = \mathbb R^*_+.$$
\end{definition}

The next few Lemmas will be useful in the proof of our main result. 

\begin{lemma} \label{label 86}
Given $3<\beta \leq 4 $ and $\omega > 1$, there exist two sequences of distinct primes $\{p_n\}_{n\geq 1}$ and $\{q\}_{n\geq 1}$ such that

\begin{equation*}
    \lim_n \frac{q_n^{\beta}}{p_n^{\beta}}=\omega, \quad \text{and} \quad \sum_n \frac{1}{p_n^{\beta-3}}=\sum_n\frac{1}{q_n^{\beta-3}}=\infty
\end{equation*}
\end{lemma}

\begin{proof}
This follows from the proof of \cite[Theorem 2.9]{boca2000factors}.
\end{proof}

\begin{lemma} \label{label 86*}
Let $3<\beta \leq 4$ and $p \in \mathcal{P}$ a prime number. Then for the operator $m(A_p)T_{g_{1,p}}m(B_p)$ acting on the space $L^2(\Gamma \backslash X,\nu_{\beta})$ we have that

\begin{equation*}
    \norm{m(A_p)T_{g_{1,p}}m(B_p)} \leq \nu_{\beta}(\Gamma_2 \backslash B_p)^{-1/2}.
\end{equation*}
\end{lemma}

\begin{proof}
It is easy to verify that $B_p = \Gamma_2 g_{1,p}A_p$. We have that $\deg_{\Gamma_2}(g_{1,p})= (1+p)(1+p^2)$, so we fix representatives $\{h_i\}_{1 \leq i\leq (1+p)(1+p^2) }$ of $\Gamma_2 \backslash \Gamma_2 g_{1,p} \Gamma_2$ and choose a fundamental domain $U$ for the action of the discreet group $\Gamma_2$ on $A_p$. We claim that the sets $\Gamma_2 h_{i} U \cap \Gamma_2 h_{j} U = \emptyset$ for $i\neq j$ and the projection map $\pi: X \rightarrow \Gamma_2 \backslash X$ is injective on the sets $h_iU$. Indeed, if $h_j^{-1}\gamma h_i x_1 =x_2$, for some $\gamma\in \Gamma_2 $ and $x_1,x_2\in A_p,$ then necessarily $h_j^{-1}\gamma h_i \in GSp_{4}(\mathbb Z_p) \cap G_p=\Gamma$. Since $\pi$ is injective on $U$, we obtain that $x_1=x_2$ and since the action of $\Gamma_2$ on $A_p$ is free, it follows that $i=j$. Given any $f\in L^2(\Gamma _2\backslash X,\nu_{\beta})$, we have that $\abs{T_g(f)}^2 \leq T_g(\abs{f}^2)$ point-wise since the function $t\mapsto t^2$ is convex. Since
\begin{equation*}
    \lambda(h_i)=p \quad \forall i= 1, \dots,(1+p)(1+p^2)
\end{equation*}
and the $ \pi(h_iU), i= 1, \dots,(1+p)(1+p^2)$ are disjoint, we obtain ()
\begin{align*}
\norm{m(A_p)T_{g_{1,p}}m(B_p)(f)}_2^2 
&= \norm{m(A_p)T_{g_{1,p}}(f)}_2^2\\
& \leq \int_{\Gamma_2 \backslash A_p}\abs{T_{g_{1,p}}(f)}^2 \d \nu_{\beta}\\
& \leq \int_{\Gamma_2 \backslash A_p}T_{g_{1,p}}(\abs{f})^2 \d \nu_{\beta}\\
& = \frac{1}{\deg_{\Gamma_2}(g_{1,p})}\sum_{i=1}^{(1+p)(1+p^2)} \int_{U}\abs{f(p(h_i \,\, \cdot)}^2  \d \mu_{\beta}\\
& = \frac{p^{\beta}}{\deg_{\Gamma_2}(g_{1,p})}\sum_{i=1}^{(1+p)(1+p^2)} \int_{h_iU}(f \circ p )^2  \d \mu_{\beta}\\
& \leq \frac{p^{\beta}}{\deg_{\Gamma_2}(g_{1,p})}\norm{f}_2^2.
\end{align*}
Thus
\begin{equation*}
     \norm{m(A_p)T_{g_{1,p}}m(B_p)} \leq p^{\beta/2}\deg_{\Gamma_2}(g_{1,p})^{-1/2}.
\end{equation*}

On the other hand (recall that the measure $\mu_\beta$ is in fact a product measure as constructed in \cite{abouamal2022bost}) we have that

\begin{equation*}
    \mu_{\beta,p}(GSp_4(\mathbb Z_p))= \zeta_{S_{2,p},\Gamma_2}(\beta)^{-1}.
\end{equation*}

Since $B_p = \Gamma_2 g_{1,p} A_p$, we can now compute $\nu_{\beta}(\Gamma_2 \backslash B_p)$ using the scaling property of $\mu_{\beta,p}$. Hence

$$\nu_{\beta}(\Gamma_2 \backslash B_p)= p^{-\beta} \deg_{\Gamma_2}(g_{1,p}) \nu_{\beta}(\Gamma_2 \backslash A_p) \leq p^{-\beta} \deg_{\Gamma_2}(g_{1,p}),$$
which concludes the proof since $\deg_{\Gamma_2}(g_{1,p})=(1+p)(1+p^2)$.
\end{proof}

\begin{lemma} \label{label 89}
Given $r\in GSp_{4}(\mathbb Q_F)$ and a finite set of primes $F$, we set
$$Z:=\Gamma_2 \backslash PGSp_{4}^+(\mathbb R) \times (GSp_4(\mathbb Z_F) r GSp_4(\mathbb Z_F) ).$$

Assume $f$ is a continuous right $GSp_{4}(\mathbb Z_F)$-invariant function on $Z$ with compact support. Then for any $\epsilon >0$, there exits a constant $C(\epsilon)$ such that for any compact subset $\Omega $ of $Z$ and any finite subset $S$ of $F^c$, we have that

\begin{equation*}
    \abs{T_{g}f(x)- \nu_{\beta,F}(Z)^{-1}\int_{Z}f \d \nu_{\beta,F}} < C(\epsilon) \,\prod_{p\in S}p ^{2\epsilon-1} \quad \text{for all}\,\,\, x\in \Omega,
\end{equation*}
where $g=\prod_{p\in S} g_{1,p}$.
\end{lemma}

\begin{proof}
We let 
$$H= GSp_{4}(\mathbb Z_F) \cap r GSp_{4}(\mathbb Z_F)r^{-1},$$ 
and

$$K= H \times \prod_{p\in F^c} GSp_{4}(\mathbb Z_p).$$

By viewing $GSp_{4}(\mathbb Z_F)$ and $\prod_{p\in F^c} GSp_{4}(\mathbb Z_p)$ as subgroups of $ GSp_4(\mathbb \hZ )$ (by considering $GSp_{4}(\mathbb Z_F)$ as the subgroup of $ GSp_4(\mathbb \hZ )$ consisting of elements with coordinates $1$ for $p \in F^c$ ), we obtain the following homeomorphism

\begin{equation*} 
\Gamma_2 \backslash PGSp^+_{4}(\mathbb R) \times  GSp_{4}(\mathbb Z_F)/ H \simeq \Gamma_2 \backslash PGSp^+_{4}(\mathbb R)\times GSp_4(\mathbb \hZ )/K.
\end{equation*}

The quotient $GSp_{4}(\mathbb Z_F)/ H$ can be unidentified with the $GSp_{4}(\mathbb Z_F)$-space \linebreak 
$GSp_4(\mathbb Z_F) r GSp_4(\mathbb Z_F) $. Hence we can consider $f$ as function on 
$$ \Gamma_2 \backslash PGSp^+_{4}(\mathbb R)\times GSp_4(\mathbb \hZ )/K. $$

Next, we have that $ GSp_{4}(\hZ) = \Gamma_2 K$. In fact, since $K$ is an open compact subgroup of $GSp_4(\hZ)$, this follows if the surjectivity of the map $\lambda: H \rightarrow \mathbb{Z}^{\times}_F$ is assumed.

Let $x \in \mathbb Z_F^{\times}$ and consider a diagonal element $\alpha \in GSp_{4}(\mathbb{Z}_F)$ such that $\lambda(\alpha) = x$. We choose $\gamma_1, \gamma_2 \in GSp_{4}(\mathbb Z_F)$ and $\tilde{r}$ a diagonal element of $GSp_4 {\mathbb (\mathbb Q_F)}$ such that $r = \gamma_1 \tilde{r} \gamma_2$ (this follows from the proof of the Elementary Divisor Theorem since the $p$-adic ring of integers is PID). Then it is clear that $\gamma_1\alpha \gamma_1^{-1} \in H$ since  $\alpha = \tilde{r} \alpha \tilde{r}^{-1}$. Since $\lambda ( \gamma_1\alpha \gamma_1^{-1} ) = \lambda(\alpha)$, we conclude that $\lambda (H)=\mathbb Z_F^{\times}$. 

We can now proceed as in the proof of \cite[Proposition 3.15]{abouamal2022bost} and use \cite[Theorem 1.7 and section 4.7]{clozel2001hecke} to obtain the upper bound.
\end{proof}

\begin{lemma} \label{label 92}

Let $F$ be a finite set of primes and $f$ be any positive continuous right $GSp_{4}(\mathbb Z_F)$-invariant function on $\Gamma_2 \backslash (PGSp_{4}^+(\mathbb R) \times MSp_{4}(\mathbb Z_F)) \subset \Gamma_2 \backslash X_F$ with $\int_{\Gamma_2 \backslash X_F} f \d \nu_{\beta,F}=1$. Then given any $0 < \delta < 1$, there exists $M>0$ such that 
\begin{equation*}
    \int_{\Gamma_2 \backslash X_{F}} (T_{g_{1,p}}f) (T_{g_{1,q}}f) \d \nu_{\beta,F} \geq (1-\delta)^5,\quad \text{for all} \,\,\, p,q>M, \quad p,q \in F^c
\end{equation*}

\end{lemma}

\begin{proof}
Fix  $0<\delta <1$ and we consider the following decomposition

\begin{equation*}
  \Gamma_2 \backslash PGS_{4}(\mathbb R) \times (GSp_{4}(\mathbb Q_F)\cap MSp_{4}(\mathbb Z_F))= \bigcup_{k\geq 1} Z_k,
\end{equation*}

where $Z_k=\Gamma_2 \backslash (PGS_{4}^+(\mathbb R) \times (GSp_{4}(\mathbb Z_F)g_k GSp_{4}(\mathbb Z_F))$ and ${(g_k)}_{k\geq 1}$ are representatives of the double coset

$$GSp_{4}(\mathbb Z_F) \backslash (GSp_{4}(\mathbb Q_F)\cap MSp_{4}(\mathbb Z_F)) / GSp_{4}(\mathbb Z_F).$$

Given any $N\in \mathbb N$ and any compact subsets $C_k$ of $Z_k$, $k=1,\dots,N$, we can use Lemma \ref{label 89} to find $M>0$ such that if $p\in F^c$ with $p>M$, then

\begin{equation*}
 \abs{T_{g_{1,p}}f(x) -  \nu_{\beta,F}(Z_k)^{-1}\int_{Z_k}f \d \nu_{\beta,F}} < \delta  \nu_{\beta,F}(Z_k)^{-1}\int_{Z_k}f \d \nu_{\beta,F},\quad \forall x\in C_k, \quad 1 \leq k\leq N.
\end{equation*}

Hence for two distinct primes $p$ and $q$ such that $p,q > M$, we get

\begin{align*}
    \int_{\Gamma_2 \backslash X_{F}} (T_{g_{1,p}}f) (T_{g_{1,q}}f) \d \nu_{\beta,F} 
    & \geq \sum_{k=1}^{N} \int_{C_k} (T_{g_{1,p}}f) (T_{g_{1,q}}f) \d \nu_{\beta,F}\\
    & \geq (1-\delta)^2 \sum_{k=1}^{N} \Big ( \int_{Z_k} f \d \nu_{\beta,F}\Big)^2 \nu_{\beta,F}(Z_k)^{-2}\nu_{\beta,F}(C_k). 
\end{align*}

By regularity of the meausre $\nu_{\beta,F}$, we can choose the compact subsets $C_k$ such that 

\begin{equation} \label{label 90}
    \nu_{\beta,F}(Z_k)-\nu_{\beta,F}(C_k)< \delta\nu_{\beta,F}(Z_k),\quad 1 \leq k \leq N.
\end{equation}

Moreover, recall that the subset $\cup_{k}Z_k \subset \Gamma_2 \backslash X_F$ has full measure, hence we choose $N$ large such that

\begin{equation} \label{label 91}
   \int_{\Gamma_2 \backslash X_F} f \d \nu_{\beta,F}- \sum_{k=1}^{N} \int_{Z_k}f \nu_{\beta,F} < \delta.
\end{equation}

Combining equations \eqref{label 90} and \eqref{label 91}, we obtain by Jensen's inequality that for any $p,q >M$, we have

\begin{align*}
     \int_{\Gamma_2 \backslash X_{F}} (T_{g_{1,p}}f) (T_{g_{1,q}}f) \d \nu_{\beta,F} 
    & \geq (1-\delta)^3
    \Big(\sum_{k=1}^{N} \nu_{\beta,F}(Z_k)\Big) \sum_{k=1}^{N}  \frac{\nu_{\beta,F}(Z_k)}{\sum_{k=1}^{N} \nu_{\beta,F}(Z_k)}\Big(\frac{1}{\nu_{\beta,F}(Z_k)} \int_{Z_k} f \d \nu_{\beta,F}\Big)^2\\
    & \geq \frac{(1-\delta)^3}{\sum_{k=1}^{N} \nu_{\beta,F}(Z_k)} \Big( \sum_{k=1}^{N} \int_{Z_k} f \d \nu_{\beta,F} \big)^2\\
    & \geq (1-\delta)^5,
\end{align*}

since $\int_{\Gamma_2 \backslash X_F} f \d \nu_{\beta,F}=1$ and $\bigcup_{k=1}^{N}Z_k \subset \Gamma_2 \backslash PGSp_{4}^+(\mathbb R) \times MSp_{4}(\mathbb{Z}_{F})$.

\end{proof}

\begin{lemma} \label{label 107}
Let $B$ be a measurable $\Gamma_2$-invariant subset of $Y$ and define $\phi \in L^2(\Gamma_2 \backslash X,d \nu_\beta)$ as follows:

$$ \phi = \nu^{-1}_\beta (\Gamma_2 \backslash B)\, \mathbb {1}_{\Gamma_2 \backslash B}.$$
Then there exists a finite set of primes $F$ and a function $f \in L^2(\Gamma \backslash X_F, d\nu_{\beta,F})$ such that 

$$ \int_{\Gamma_2 \backslash X_F} f d\nu_{\beta, F} = 1,$$
and 
$$\norm{f_F - \phi}_2 \rightarrow 0 \quad \textmd{as} \quad F \nearrow \mathcal{P}.$$

\end{lemma}
\begin{proof}
Let 

$$ f :=  \nu^{-1}_{\beta,F} (\Gamma_2 \backslash \pi_{F}(B) )\, \mathbb {1}_{\Gamma_2 \backslash \pi_F(B) }. $$

Hence
\begin{align*}
\int_{\Gamma_2 \backslash X} \abs{f_F}^2 d \nu_{\beta} 
& = \int_{\Gamma_2 \backslash Y} \abs{f \circ \pi_F}^2 d \nu_{\beta}\\
& = \nu_{\beta,F}^{-1}(\Gamma_2 \backslash \pi_F(B)).
\end{align*}

On the other hand we have
\begin{equation*}
\int_{\Gamma_2 \backslash X} \abs{\phi}^2 d \nu_{\beta} = \nu^{-1}_{\beta}(\Gamma_2 \backslash B).
\end{equation*}

Hence $\norm{f_F}_2 \rightarrow \norm{\phi}_2$ as $F \nearrow \mathcal{P}$, which concludes the proof since $(f_F, \phi) = \norm{\phi}_2 $.  

\end{proof}

\par For the next Lemma, we use the following notation. Given two sequences $\{a_n\}$ and $\{b_n\}$, we write $a_n \sim b_n$ if $\lim_{n} (a_n/b_n) =1$ and $$\sum_n a_n \sim \sum_n b_n$$ if the two series are simultaneously divergent or convergent;

\begin{lemma} \label{label 85}
Let $\beta,\omega \in \mathbb R_+^*$ such that $3< \beta \leq 4$ and $\omega > 1$ and  set

\begin{equation*}
    \kappa := \frac{\omega^{(3-\beta) / 2\beta}}{1+\omega^{(3-\beta)/\beta}}. 
\end{equation*}

Then given any finite set of primes $F$ and any positive continuous right $GSp_{4}(\mathbb Z_F)$-invariant function on $\Gamma_2 \backslash (PGSp_{4}^+(\mathbb R) \times MSp_{4}(\mathbb Z_F))$ with $\int_{\Gamma_2 \backslash X_F} f \d \nu_{\beta,F}=1$, there exist two sequences of distinct primes $\{p_n\}_{n\geq 1}$ and $\{q_n\}_{n\geq 1}$ in $F^c$ and $\Gamma_2$-invariant measurable subsets $X_{1n},X_{2n},Y_{1n}$ and $Y_{2n}$, $n\geq 1$ of $X$ such that:
\begin{enumerate}
    \item $ \lim_{n} \abs{q_n^{\beta} / p_n^{\beta} -\omega}=0$
    \item The sets $Y_{1n}$ and $Y_{2n}, n\geq 1$ are mutually disjoint;
    \item $\sum_{n=1}^{\infty}\Big(\frac{m(X_{1n})T_{g_{n}} m(Y_{1n})}{\norm{m(X_{1n})T_{g_{n}} m(Y_{1n})}}f_F , \frac{m(X_{2n})T_{h_{n}} m(Y_{2n})}{\norm{m(X_{2n})T_{h_{n}} m(Y_{2n})}}f_F\Big) \geq \kappa$ where $g_n:=g_{1,p_n}$ and $h_n:=g_{1,q_n}$.
\end{enumerate}
\end{lemma}
\begin{proof}
Let $F\subset \mathcal{P}$ be any nonempty finite set of primes and $f$ any positive continuous right $GSp_{4}(\mathbb Z_F)$-invariant function on $\Gamma_2 \backslash (PGSp_{4}^+(\mathbb R) \times MSp_{4}(\mathbb Z_F))$ with $\int_{\Gamma_2 \backslash X_F} f \d \nu_{\beta,F}=1$ and fix $\epsilon >0$. By Lemma \ref{label 86} we can find two disjoint sequences of prime numbers $\{p_n\}_{n\geq 1}$ and $\{q_n\}_{n\geq 1}$ in $F^c$ such that $$\lim_{n}q_n^\beta/p_n^\beta = \omega,$$ 

and

\begin{equation} \label{label 93}
    \sum_{n=1}^{\infty}\frac{1}{p_n^{\beta-3}}=\infty.
\end{equation}

We let $B_n^{(1)}=\cup_{k=1}^{k={n-1}}B_{p_k}$ and $B_n^{(2)}=\cup_{k=1}^{k={n-1}}B_{q_k}$  (where $B_{p_k}$ and $B_{q_k}$ are as in \eqref{label 109}) and set

\begin{align*}
X_{1n}:= A_{p_n} \backslash B_n^{(1)},\quad Y_{1n}:=B_{p_n}\backslash B_n^{(1)},\\ X_{2n}:=A_{q_n}\backslash B_n^{(2)},\quad Y_{2n}:=B_{q_n} \backslash B_n^{(2)},    
\end{align*}

where $A_{p_n},A_{q_n}$ are as in \eqref{label 108}. By construction the sets $Y_{1n}$ and $Y_{2n},n\geq 1$ are mutually disjoint so it remains to show the last assertion. By Lemma \ref{label 92}, we choose $M >0$ and the sequences $\{p_n\}_{n\geq 1},\{q_n\}_{\geq 1}$ such that

\begin{equation} \label{label 88}
    \int_{\Gamma_2 \backslash X_{F}} (T_{g_n}f) (T_{h_n}f) \d \nu_{\beta,F} \geq (1-\epsilon)^{1/2},\quad \forall n\geq 1
\end{equation}

Observe that if $g\in \Gamma_2 g_{1,p_n} \Gamma_{2}$ then $gX_{1n}\subset Y_{1n}$ since $gA_{p_n} \subset B_{p_n}$ and $\abs{\lambda(g)}_{p_{k}}=p_k^{-1}$ for all $1\leq k<n$. By definition of the Hecke operator $T_{ g_{_{1,p_n}} }$ we get that $$m(X_{1n}) T_{g_{_{1,p_n}}}m(Y_{1n})f_F = m(X_{1n}) (T_{g_{_{1,p_n}}}f)_F.$$
Similarly, we have

$$m(X_{2n}) T_{g_{_{1,q_n}} }m(Y_{2n})f_F = m(X_{2n}) (T_{g_{_{1,q_n}}}f)_F.$$

By Lemma \ref{label 86*} and equation \eqref{label 88} we obtain

\begin{align*}
 &\sum_{n=1}^{\infty} \Big(\frac{m(X_{1n})T_{g_{n}} m(Y_{1n})}{\norm{m(X_{1n}) T _{g_{n}} m(Y_{1n})}}f_F , \frac{m(X_{2n})T_{h_{n}} m(Y_{2n})}{\norm{m(X_{2n})T_{h_{n}} m(Y_{2n})}}f_F\Big) \\
 & \geq \sum_{n=1}^{\infty} (\nu_{\beta} (\Gamma_2 \backslash B_{p_n}) \nu_{\beta} (\Gamma_2 \backslash B_{q_n}))^{1/2} \nu_{\beta}(\Gamma_2 \backslash X_{1n}\cap X_{2n}) \int_{\Gamma_2 \backslash X_{F}} (T_{g_n}f) (T_{h_n}f) \d \nu_{\beta,F}\\
 & \geq \sum_{n=1}^{\infty} (\nu_{\beta} (\Gamma_2 \backslash B_{p_n}) \nu_{\beta} (\Gamma_2 \backslash B_{q_n}))^{1/2} \Big(\prod_{k=1}^{n-1}(1-\nu_{\beta} (\Gamma_2 \backslash B_{p_k} \cup B_{q_k})) \Big) \nu_\beta(\Gamma_2 \backslash A_{p_n})\nu_\beta(\Gamma_2 \backslash A_{q_n}) (1-\epsilon)^{1/2}.
\end{align*}

Since 

\begin{equation*}
    \nu_\beta(\Gamma_2 \backslash A_{p_n})\nu_\beta(\Gamma_2 \backslash A_{q_n}) = \zeta_{S_{2,p_n},\Gamma_2}(\beta)^{-1} \zeta_{S_{2,q_n},\Gamma_2}(\beta)^{-1},
\end{equation*}
and 
\begin{equation*}
    \nu_{\beta}(\Gamma_2 \backslash B_{p_n}) \nu_{\beta}(\Gamma_2 \backslash B_{q_n}) = (p_nq_n)^{-\beta} \deg_{\Gamma_2}(g_{1,p_n})\deg_{\Gamma_2}(g_{1,q_n}) \zeta_{S_{2,p_n},\Gamma_2}(\beta)^{-1}\zeta_{S_{2,q_n},\Gamma_2}(\beta)^{-1},
\end{equation*}
we obtain that

\begin{equation*}
    \frac{(\nu_{\beta} (\Gamma_2 \backslash B_{p_n}) \nu_{\beta} (\Gamma_2 \backslash B_{q_n}))^{1/2}  \nu_\beta(\Gamma_2 \backslash A_{p_n})\nu_\beta(\Gamma_2 \backslash A_{q_n})} { \nu_{\beta}(\Gamma_2 \backslash B_{p_n}\cup B_{q_n})} \sim  \frac{(p_n^{\beta} q_n^{\beta})^{3-\beta / 2\beta}}{(p_n^{3-\beta} + q_n^{3-\beta} - (p_nq_n)^{3-\beta})},
\end{equation*}

since $3 < \beta \leq 4 $.
Hence we can choose the sequences $\{p_n\}_{n\geq 1}$ and $\{q_n\}_{\geq 1}$ such that

\begin{equation*}
    \frac{(\nu_{\beta} (\Gamma_2 \backslash B_{p_n}) \nu_{\beta} (\Gamma_2 \backslash B_{q_n}))^{1/2}  \nu_\beta(\Gamma_2 \backslash A_{p_n})\nu_\beta(\Gamma_2 \backslash A_{q_n})} { \nu_{\beta}(\Gamma_2 \backslash B_{p_n}\cup B_{q_n})} >\frac{\omega^{(3-\beta) / 2\beta}}{1+\omega^{(3-\beta)/\beta}} (1-\epsilon)^{1/2}, \quad \forall n\geq 1
\end{equation*}

Since 

\begin{equation*}
    \sum_{n=1}^{\infty} \nu_{\beta} (\Gamma_2 \backslash B_{p_n} \cup B_{q_n}) \geq \sum_{n=1}^{\infty} \nu_{\beta} (\Gamma_2 \backslash B_{p_n}) \sim \sum_{n=1}^{\infty} {\frac{1}{p_n^{\beta-3}}}= \infty
\end{equation*}

by equation \eqref{label 93}, we finally obtain that

\begin{equation*}
    \sum_{n=1}^{\infty} \Big(\frac{m(X_{1n})T_{g_{n}} m(Y_{1n})}{\norm{m(X_{1n}) T _{g_{n}} m(Y_{1n})}}f_F , \frac{m(X_{2n})T_{h_{n}} m(Y_{2n})}{\norm{m(X_{2n})T_{h_{n}} m(Y_{2n})}}f_F\Big) \geq \frac{\omega^{(3-\beta) / 2\beta}}{1+\omega^{(3-\beta)/\beta}} (1-\epsilon),
\end{equation*}

where the last inequality follows from the fact that 
$$ \sum_{n =1}^{\infty} \nu_{\beta}(\Gamma_2 \backslash B_{p_n}\cup B_{q_n}) \Big(\prod_{k=1}^{n-1}(1-\nu_{\beta} (\Gamma_2 \backslash B_{p_k} \cup B_{q_k})) \Big) = 1$$

Since $\epsilon$ was arbitrary, this completes the proof.
\end{proof}

Denote by $\lambda_\infty$ the Lebesgue measure on $\mathbb R$. We have three commuting actions of $G$, $\mathbb R$ and $GSp_{4}(\hZ)$ on the space $(\mathbb R_+ \times X, \lambda_\infty \times \mu_\beta)$ as follows:
\begin{align}
    g(t,x) &= \Big(\frac{d \mu \circ \alpha_g}{\d \mu} (gx)\, t, gx \Big) \quad \textmd{for} \,\,\,  g\in G, \label{extended_1}\\
    s(t,x) &= (e^{-s}t,x) \quad \textmd{for} \,\,\, s\in \mathbb R,     \quad
    h (t,x) = (t, hx ) \quad \textmd{for}\,\,\, h \in GSp_{4}(\hZ). \label{extented_2}
\end{align}

\begin{prop} \label{label 112}
   If the action of $G$ on $(X/ GSp_{4}(\hZ), \nu_\beta )$ is of type $\textmd{III}_1$ then the action of $G$ on $(\mathbb R_+ \times X, \lambda_\infty \times \mu)$ is ergodic.
\end{prop}

\begin{proof}
The assumption together with \cite[Theorem 3.16]{abouamal2022bost} and the characterization of type $\textmd{III}_1$ action in terms of the extended action of $G$ given in \eqref{extended_1} (See \cite{neshveyev2011neumann})
) we obtain that:
 $$ L^{\infty} (X, \mu_\beta ) ^{G} = \mathbb C, \quad L^{\infty} (\mathbb R_+ \times X/ GSp_{4}(\hZ), \lambda_\infty \times  \mu_\beta ) ^{G} = \mathbb C.$$

 The result follows from \cite[Propositon 4.6]{laca2007phase} since the actions of $G$, $\mathbb R$ and $GSp_{4}(\hZ)$ on the space $(\mathbb R_+ \times X, \lambda_\infty \times \mu_\beta)$ commute, $GSp_{4}(\hZ)$ is profinite and $\mathbb R$ is connected.
\end{proof}

We are now ready to prove the main result of this paper.\\

\begin{theorem} \label{main_theorem}
For $ 3 <\beta \leq 4$, the unique $\textmd{KMS}_\beta$ state on the $GSp_{4}$-system is of type $\textmd{III}_1$.
\end{theorem}

\begin{proof}
In view of the isomorphism in \eqref{label 110} and Theorem \ref{label 111}, we need to show that the action of $G$ on $( X, \mu_\beta)$ is of type $\textmd{III}_1$. This is the case if and only if the action of $G$ on $(\mathbb R_+ \times X, \lambda_\infty \times \mu)$ is ergodic. Hence by Proposition \ref{label 112} it is enough to show that the action of $GSp_{4}^+(\mathbb{Q})$ on the space 
$(PGSp_{4}^+(\mathbb R) \times MSp_4(\mathbb{A}_f) / GSp_{4}(\hZ),\nu_\beta)$ is of type $\textmd{III}_1$. Since $ r(\mathcal R,\mu)^*$ is a closed subgroup of $\mathbb R_{+}^*$, it is enough to show that any real number $\omega> 1$ belongs to the ratio set $r(\mathcal{R},\mu_\beta)$ corresponding to this action. Fix $\epsilon > 0$ and let $B$ be any measurable right $GSp_{4}(\hZ)$-invariant subset of $X$ with positive measure. 

Let $F$ be a finite non-empty set of primes, $f$ any positive continuous right $GSp_{4}(\mathbb Z_F)$-invariant function with compact support in $\Gamma_2 \backslash (PGSp_{4}^+(\mathbb R) \times MSp_4{(\mathbb Z _F)})$, $X_{1n},X_{2n}, Y_{1n},Y_{2n}$ any mutually disjoint $\Gamma_2$-invariant measurable subsets of $X$ and $\{p_n\}_{n\geq 1},\{q_n\}_{n\geq 1}$ any two sequences of distinct primes in $F^c$. To ease notation we set

$$ T_n^{(1)}= \frac{m(X_{1n})T_{g_{n}} m(Y_{1n})}{ \norm{m(X_{1n})T_{g_{n}} m(Y_{1n})}},\quad T_n^{(2)}= \frac{m(X_{2n})T_{h_{n}} m(Y_{2n})}{\norm{m(X_{2n})T_{h_{n}} m(Y_{2n})}},$$

$$ e_n^{(1)}:=m(Y_{1n}),\quad e_n^{(2)}:=m(Y_{2n}).$$

Let $\phi \in L^2(\Gamma_2 \backslash X, d\nu_\beta)$. Since $\norm{T_n^{(1)}}=\norm{T_n^{(2)}}=1$ and $e_n',e_n''$ are projections, we obtain by Cauchy-Schwartz that

\begin{align*}
    \sum_{n} (T_n^{(1)}\phi,T_n^{(2)}\phi) 
    &\geq \sum_{n} (T_n^{(1)} f_F,T_n^{(2)} f_F) - \norm{e_n^{(1)}(f_F -\phi) }_2 \norm{e^{(2)}_nf_F}_2 -
    \norm{e_n^{(2)}(f_F -\phi) }_2 \norm{e^{(1)}_n\phi}_2 \\
    & \geq  \sum_{n} (T_n^{(1)} f_F,T_n^{(2)} f_F) - \Big( \sum_{n} \norm{e^{(1)}_n(f_F - \phi)}_2^2\Big)^{1/2} \Big(\sum_{n}  \norm{e_n^{(2)} f_F}_2^2 \Big)^{1/2} \\
    &- \Big( \sum_{n}  \norm{e^{(2)}_n(f_F - \phi)}_2^2 \Big)^{1/2}
    \Big( \sum_{n}  \norm{e_n^{(1)} (\phi)}_2^2  \Big)^{1/2}\\
    & \geq  \sum_{n} (T_n^{(1)} f_F,T_n^{(2)} f_F) - \norm{f_F -\phi}_2 (\norm{f_F}_2+ \norm{\phi}_2 ).
\end{align*}

Since the subset $GSp_{4}^+(\mathbb Q)B$ is completely determined by its intersection with $PGSp_{4}^+(\mathbb R) \times MSp_{4}(\hZ)$, there exists $g_0$ such that the intersection $B_0:= g_0B \cap ( PGSp_{4}^+(\mathbb R) \times MSp_{4}(\hZ))$ has positive measure. We set $$\phi:= \nu_\beta(\Gamma_2 \backslash \Gamma_2 B_0) \mathbb{1}_{\Gamma_2 \backslash \Gamma_2 B_0}.$$

Let $\kappa = \frac{\omega^{(3-\beta) / 2\beta}}{1+\omega^{(3-\beta)/\beta}}$. By Lemma \ref{label 107} there exists $f$ and $F \subset \mathcal{P}$ large enough such that 
$$\norm{f_F -\phi}_2 (\norm{f_F}_2+ \norm{\phi}_2 ) <\kappa, \quad \int_{\Gamma_2 \backslash X_F} f d\nu_{F,\beta} = 1.$$

Hence by Lemma \ref{label 85} there exists $m\in \mathbb N$ such that $(T_{m}^{(1)}\phi,T_{m}^{(2)}\phi)>0$. This implies that $(T_{g_m}\phi,T_{h_m}\phi)>0$, in particular this shows that the subset $\Gamma_2 g_m^{-1} \Gamma_2 B_0 \cap \Gamma_2 h_m^{-1} \Gamma_2 B_0 \subset X$ has positive measure.
Thus there exist $g\in \Gamma_2 g_m \Gamma_2$ and $h \in \Gamma_2 h_m \Gamma_2$ such that $g^{-1}B_0 \cap h^{-1} B_0$ has positive measure, which implies that the set $g_0^{-1}h g^{-1}g_0B \cap B$ has positive measure. If we set $\tilde{g}:=g_0^{-1}h g^{-1}g_0$, we get by the scaling condition that

\begin{equation*}
    \abs{\frac{ \d \tilde{g}_* \mu_\beta }{\d \mu_\beta} (x) -\omega} =   \abs{\lambda(g_0^{-1}h g^{-1}g_0)^\beta -\omega} = \abs{\frac{q_m^\beta }{p_m^\beta }-\omega} < \epsilon,\quad  \forall x\in \tilde{g}B\cap B. 
\end{equation*}

This shows that $\omega\in r(\mathcal{R},\mu_\beta)$, which completes the proof.
\end{proof}

We conclude this paper by the following Theorem. It summarizes the full thermodynamics of the Connes-Marcolli $GSp_4$-system.

    \begin{theorem}   
The $GSp_{4}$-system has the following properties:
\begin{enumerate}

    \item There is no $\textmd{KMS}_\beta$ state in the range $ 0 <\beta <3$ and $\beta \notin \{1,2\}$.
    \item There exists a unique $\textmd{KMS}_\beta$ state in the the range $3 <\beta \leq 4 $. Moreover, this state is of type $\textmd{III}_1$
    \item In the range $4<\beta \leq \infty $, the set of extremal
    states is identified with the Shimura variety $\textmd{Sh}(GSp_{4},\mathbb H_2^{\pm})$, $$\mathcal{E}_\beta \simeq GSp_{4}(\mathbb Q)\backslash \mathbb H_{2}^{\pm} \times GSp_{4}(\mathbb A_{\mathbb Q,f}).$$
    \noindent The explicit expression of the extremal $\textmd{KMS}_{\beta}$ states is given by 
    
    \begin{equation} 
    \phi_{\beta,y}(f)=\frac{\zeta(2\beta-2)\textmd{Tr} (\pi_{y}(f)e^{-\beta H_y})}{\zeta(\beta) \zeta(\beta-1)\zeta(\beta-2) \zeta(\beta-3)},\quad y \in  \mathbb H_2 ^+ \times GSp_{4}(\hZ) ,\quad \forall f\in \mathcal{A}.
    \end{equation}

 \noindent Every such a state is of type $\textmd{I}_\infty$
    
\end{enumerate}
\end{theorem}

\begin{remark}

  The analysis of the $GSp_4$-system is closely related to the structure of the Hecke pair $(\Gamma_{2n}, GSp_{2n}^+(\mathbb Q))$, which is less explicit for $n \geq 2$. In our case $n = 2$, we were able to derive approximate formulas for $\textmd{deg}_{\Gamma_2}(g)$ given an arbitrary element $g\in GSp_{2n}^+(\mathbb Q)$. Moreover, in some  key Lemmas, we were able to carry on the analysis by using specific matrices so that a closed formula for $\textmd{deg}_{\Gamma_2}(g)$ can be used. This approach will not be possible in the general case $n > 2$. The author believes that it is still possible to extend the results of this paper and \cite{abouamal2022bost} to the general case $GSp_{2n}, n > 2$. More precisely, we conjecture that for $n > 2$, a phase transition occurs at $\beta = n (n + 1)/2$ and $\beta = 2n$ and that there are no $\textmd{KMS}_\beta$ states for $\beta < n (n + 1)/2$.
\end{remark}

\bibliographystyle{plain}
\bibliography{refs.bib}

\end{document}